\documentclass[12pt,reqno]{article}
\usepackage{amsfonts}
\usepackage{mathrsfs}

\usepackage{amsmath,amsthm,amsfonts,amssymb,latexsym,mathrsfs,color}

\usepackage[colorlinks=true,
linkcolor=webblue, filecolor=webbrown, citecolor=webred ]{hyperref}
\definecolor{webblue}{rgb}{0,.5,0}
\definecolor{webred}{rgb}{0,.5,0}
\definecolor{webbrown}{rgb}{.6,0,0}

\newtheorem{thm}{Theorem}[section]

\newtheorem{cor}[thm]{Corollary}
\newtheorem{prop}[thm]{Proposition}

\newtheorem{ex}[thm]{Example}

\newtheorem{f}{Fact}[section]

\theoremstyle{definition}

\theoremstyle{remark}
\newtheorem{rem}[thm]{Remark}

\numberwithin{equation}{section}

\title{Linear transformations and strong $q$-log-concavity for certain combinatorial triangle
\thanks{Supported partially by the National Natural Science Foundation of China (Grant No. 11571150).
\newline\hspace*{5mm}
   {\it Email address:}\quad
    bxzhu@jsnu.edu.cn (B.-X. Zhu)}}
\author{Bao-Xuan Zhu}
\date{\footnotesize School of Mathematics and Statistics,
         Xuzhou Normal University,
         Xuzhou 221116, PR China}

\begin{document}

\maketitle

\begin{abstract}
It is well-known that the binomial transformation  preserves the
log-concavity property and log-convexity property. Let
$\binom{a+n}{b+k}$ be the binomial coefficients and $\binom{n,k}{j}$
be defined by
$(b_0+b_1x+\cdots+b_kx^{k})^n:=\sum_{j=0}^{kn}\binom{n,k}{j}x^j,$
where the sequence $(b_i)_{0\leq i\leq k}$ is log-concave. In this
paper, we prove that the linear transformation
$$y_n(q)=\sum_{k=0}^n\binom{a+n}{b+k}x_k(q)$$ preserves the strong $q$-log-concavity
property for any fixed nonnegative integers $a$ and $b$, which
strengthens and gives a simple proof of results of Ehrenborg and
Steingrimsson, and Wang, respectively, on linear transformations
preserving the log-concavity property. We also show that the linear
transformation
$$y_n=\sum_{i=0}^{kn}\binom{n,k}{j}x_i$$ not only preserves the log-concavity property, but also preserves the log-convexity
property, which extends the results of Ahmia and Belbachir about the
$s$-triangle transformation preserving the log-convexity property
and log-concavity property. Let $[A_{n,k}(q)]_{n, k\geq0}$ be an
infinite lower triangular array of polynomials in $q$ with
nonnegative coefficients satisfying the recurrence
\begin{eqnarray*}\label{re}
A_{n,k}(q)=f_{n,k}(q)\,A_{n-1,k-1}(q)+g_{n,k}(q)\,A_{n-1,k}(q)+h_{n,k}(q)\,A_{n-1,k+1}(q),
\end{eqnarray*}
for $n\geq 1$ and $k\geq 0$, where $A_{0,0}(q)=1$,
$A_{0,k}(q)=A_{0,-1}(q)=0$ for $k>0$. We present criterions for the
strong $q$-log-concavity of the sequences in each row of
$[A_{n,k}(q)]_{n, k\geq0}$. As applications, we get the strong
$q$-log-concavity or the log-concavity of the sequences in each row
of many well-known triangular arrays, such as the Bell polynomials
triangle, the Eulerian polynomials triangle and the Narayana
polynomials triangle in a unified approach.
\end{abstract}

\section{Introduction}
\hspace*{\parindent}
Let $(a_n)_{n\geq0}$ be a sequence of nonnegative real numbers. It
has no {\it internal zeros }, {\it i.e.}, $a_j\neq0$ if
$a_ia_k\neq0$ for $i< j< k$. Thoughout this paper, all sequences
have no internal zeros. We call it {\it log-concave} if
$a_{k-1}a_{k+1}\le a_k^2$ for all $k\ge 1$. The log-concave
sequences arise often in combinatorics, algebra, geometry, analysis,
probability and statistics and have been extensively investigated,
see Stanley~\cite{Sta89} and Brenti~\cite{Bre94} for details.

For two polynomials with real coefficients $f(q)$ and $g(q)$, denote
$f(q)\geq_q g(q)$ if the difference
$f\left(q\right)-g\left(q\right)$ has only nonnegative coefficients.
For a polynomial sequence $(f_n(q))_{n\geq 0}$, it is called {\it
$q$-log-concave} suggested by Stanley if
$$f_n(q)^2\geq_q f_{n+1}(q)f_{n-1}(q),$$ for $n\geq 1$ and is called {\it strongly $q$-log-concave} introduced by Sagan if
$$f_{n+1}(q)f_{m-1}(q)\geq_q f_n(q)f_m(q),$$
for any $m\geq n\geq1$. Obviously, the strong $q$-log-concavity
implies the $q$-log-concavity.  The q-log-concavity of polynomials
have been extensively studied, see Butler \cite{But90},
Krattenthaler \cite{Kra}, Leroux \cite{Le90}, Sagan
\cite{Sag92,Sag921}, and Su, Wang and Yeh \cite{SWY11} for instance.
If we reverse the inequalities in above definitions, then we have
the concepts for log-convexity, $q$-log-convexity and strong
$q$-log-convexity, respectively. Reader can refer to Chen {\it et
al.} \cite{CWY10,CTWY10,CWY11}, Liu and Wang \cite{LW07}, Zhu
\cite{Zhu13,Zhu14}, and Zhu and Sun \cite{ZS15} for the strong
$q$-log-convexity.

The linear transformations are often used to study the log-concavity
and log-convexity. For instance, it was proved that the binomial
transformation $b_n=\sum_{k=0}^n\binom{n}{k}a_k$ for $n\geq0$
preserves the log-concavity property, {\it i.e.}, log-concavity of
$(a_n)_{n\geq 0}$ implies that of $(b_n)_{n\geq 0}$ (see
\cite[Theorem 2.5.7]{Bre89} or \cite[Theorem 7.3]{Kar68} for
instance) and log-convexity property \cite{LW07}. More generally,
the log-convexity property and log-concavity property are preserved
under the binomial convolution, see Davenport and
P\'olya~\cite{DP49} and Wang and Yeh~\cite{WY07}, respectively.
Using certain positivity method, Ehrenborg and Steingrimsson
\cite{ES00} also showed that the linear transformation
$$y_n=\sum_{k=0}^n\binom{n+1}{k}x_k$$ preserves the log-concavity
property. Furthermore, using the algebraical method, Wang \cite{W03}
proved one more general result that the linear transformation
$$y_n=\sum_{k=0}^n\binom{a+n}{b+k}x_k$$ preserves the log-concavity
property for fixed nonnegative integers $a$ and $b$.  Motivated by
these, we prove the following stronger result.
\begin{thm}\label{thm+binom}
Let $a$ and $b$ be two
nonnegative integers. Then the linear transformation
$$y_n(q)=\sum_{k=0}^n\binom{a+n}{b+k}x_k(q)$$ preserves the strong $q$-log-concavity
property. In particular, it preserves the log-concavity property of
sequences.
\end{thm}
The $s$-triangle $\binom{n}{j}_s$ is a generalization of the Pascal
triangle, which is given by the ordinary multinomial coefficients
\cite{Com74}:
$$(1+x+\cdots+x^{s})^n:=\sum_{j=0}^{sn}\binom{n}{j}_sx^j.$$
In \cite{AB12,AB10}, Ahmia and Belbachir also demonstrated that the
log-convexity property and log-concavity property are preserved
under the $s$-triangle transformation. Motivated by this, we will
consider a more general triangle in the following. Let $k$ be a
positive integer. Assume that the nonnegative sequence $(b_i)_{0\leq
i\leq k}$ is log-concave. Define a more generalized triangle
$\binom{n,k}{j}$:
$$(b_0+b_1x+\cdots+b_kx^{k})^n:=\sum_{j=0}^{kn}\binom{n,k}{j}x^j.$$
It is obvious that the triangle $[\binom{n,k}{j}]_{n,j\geq0}$ has
the following recurrence relations:\begin{eqnarray}
\binom{n,k}{j}&=&\sum_{i=0}^k\binom{n-1,k}{j-i}b_i,\label{recur+1}\\
\binom{n,k}{j}&=&\sum_{i=0}^{sk}\binom{n-s,k}{i}\binom{s,k}{j-i}.\label{recur+s}
\end{eqnarray}
In fact, the triangle $[\binom{n,k}{j}]_{n,j\geq0}$ generalizes many
famous triangles. For instance,
\begin{itemize}
\item[\rm (i)]
If both $k=b_0=b_1=1$, then the triangle
$[\binom{n,k}{j}]_{n,j\geq0}$ turns out be the Pascal triangle.
\item[\rm (ii)]
If $b_0=b_1=\ldots=b_k=1$, then the triangle
$[\binom{n,k}{j}]_{n,j\geq0}$ turns out be the s-triangles given by
the ordinary multinomials, see \cite[A027907 for s = 2, A008287 for
s = 3 and A035343 for s = 4]{Sl}.
\end{itemize}
We also prove the next stronger result in a unified approach.
\begin{thm}\label{thm+three}
Let $k$ be any fixed positive integer and $\binom{n,k}{j}$ be as
above. Then the linear transformation
$$y_n=\sum_{i=0}^{kn}\binom{n,k}{j}x_i$$ not only preserves the log-concavity property, but also preserves the log-convexity property.
\end{thm}

In \cite{Zhu13}, Zhu defined a triangular array as follows. Let
$[A_{n,k}(q)]_{n, k\geq 0}$ be an infinite lower triangular array
defined by the recurrence
\begin{eqnarray}\label{rre}
A_{n,k}(q)=f_{k}(q)\,A_{n-1,k-1}(q)+g_{k}(q)\,A_{n-1,k}(q)+h_{k}(q)\,A_{n-1,k+1}(q)
\end{eqnarray}
for $n\geq 1$ and $k\geq 0$, where $A_{0,0}(q)=1$,
$A_{0,k}(q)=A_{0,-1}(q)=0$ for $k>0$. This triangular array consists
of polynomials in $q$ and it will turn out to be the array with
entries be real numbers for any fixed $q\geq0$. In fact, this
triangular array (\ref{rre}) unifies many well-known combinatorial
triangles. The following are some basic examples.

\begin{ex}\label{exm}
(1) The Catalan triangle of Aigner \cite{Aig991} is
$$C=[C_{n,k}]_{n,k\geq0}=\left[
      \begin{array}{rrrrr}
        1 &  &  &  &  \\
        1 & 1 &   &   &\\
        2 & 3 & 1 &   & \\
        5 & 9 & 5 & 1 &   \\
       \vdots &  &  &  & \ddots \\
      \end{array}
    \right],$$
where $C_{n+1,k}=C_{n,k-1}+2\,C_{n,k}+C_{n,k+1}$ and
$C_{n+1,0}=C_{n,0}+C_{n,1}$. The numbers in the $0$th column are the Catalan numbers $C_n$. \\
(2) The Catalan triangle of Shaprio \cite{Sha76} is
$$C'=[C'_{n,k}]_{n,k\geq0}=\left[
      \begin{array}{rrrrr}
        1 &  &  &  &  \\
        2 & 1 &   &   &\\
        5 & 4 & 1 &   & \\
        14 & 14 & 6 & 1 &   \\
       \vdots &  &  &  & \ddots \\
      \end{array}
    \right],$$
where $C'_{n+1,k}=C'_{n,k-1}+2\,C'_{n,k}+C'_{n,k+1}$ for $k\geq0$. The numbers in the $0$th column are the Catalan numbers $C_n$. \\
(3) The Motzkin triangle \cite{Aig98,Aig991} is
  $$M=[M_{n,k}]_{n,k\geq0}=\left[
  \begin{array}{rrrrr}
                   1 &  &  &  &  \\
                   1 & 1 &  &  &  \\
                   2 & 2 & 1 &  &  \\
                   4 & 5 & 3 & 1 &  \\
                   \vdots &  &  &  & \ddots \\
                 \end{array}
  \right],$$
where $M_{n+1,k}=M_{n,k-1}+M_{n,k}+M_{n,k+1}$ and
$M_{n+1,0}=M_{n,0}+M_{n,1}$. The numbers in the $0$th column are the Motzkin numbers $M_n$.\\
(4) The large Schr\"{o}der triangle \cite{CKS12} is
 $$s=[s_{n,k}]_{n,k\geq0}=\left[
 \begin{array}{rrrrr}
 1 &  &  &  &  \\
 2 & 1 &  &  &  \\
 6 & 4 & 1 &  &  \\
 22 & 16 & 6 & 1 &  \\
 \vdots &  &  &  & \ddots \\
 \end{array}
 \right],$$ where $s_{n+1,k}=s_{n,k-1}+2\,s_{n,k}+2s_{n,k+1}$ and
$s_{n+1,0}=s_{n,0}+2s_{n,1}$. The numbers in the $0$th column are the large Schr\"{o}der numbers $S_n$. \\
(5) The Bell triangle \cite{Aig99} is
  $$B=[B_{n,k}]_{n,k\geq0}=\left[
  \begin{array}{lllll}
                   1 &  &  &  &  \\
                   1 & 1 &  &  &  \\
                   2 & 3 & 1 &  &  \\
                   5 & 10 & 6 & 1 &  \\
                   \vdots &  &  &  & \ddots \\
                 \end{array}
  \right],$$
  where $B_{n+1,k}=B_{n,k-1}+(1+k)\,B_{n,k}+(1+k)\,B_{n,k+1}$ and
$B_{n+1,0}=B_{n,0}+B_{n,1}$. The numbers in the $0$th column are the
Bell numbers.\\
(6) The Bell polynomials triangle $B=[B_{n,k}]_{n,k\geq0}$
\cite{Aig01} is
  $$\left[
  \begin{array}{lllll}
                   1 &  &  &  &  \\
                   q & 1 &  &  &  \\
                   q^2+q & 2q+1 & 1 &  &  \\
                   q^3+3q^2+q & 3q^2+6q+1 & 3q+3 & 1 &  \\
                   \vdots &  &  &  & \ddots \\
                 \end{array}
  \right],$$
  where $B_{n+1,k}=B_{n,k-1}+(q+k)\,B_{n,k}+q(1+k)\,B_{n,k+1}$. The numbers in the $0$th column are the Bell polynomials. In particular, it reduces to the Bell triangle
  \cite{Aig99}
  for $q=1$.\\
(7) The Eulerian polynomials triangle $E=[E_{n,k}]_{n,k\geq0}$
\cite{Aig01} is
  $$\left[
  \begin{array}{lllll}
                   1 &  &  &  &  \\
                   1 & 1 &  &  &  \\
                   q+1 & q+3 & 1 &  &  \\
                   q^2+4q+1 & q^2+10q+7 & 3q+6 & 1 &  \\
                    \vdots &  &  &  & \ddots \\
                 \end{array}
  \right],$$
  where $E_{n+1,k}=E_{n,k-1}+(kq+k+1)\,E_{n,k}+(k+1)^2q\,E_{n,k+1}$. The numbers in the $0$th column are
the Eulerian polynomials.\\
(8) The Narayana polynomials triangle $N=[N_{n,k}]_{n,k\geq0}$
\cite{Aig01} is
  $$\left[
  \begin{array}{lllll}
                   1 &  &  &  &  \\
                   q & 1 &  &  &  \\
                   q^2+q & 2q+1 & 1 &  &  \\
                   q^3+3q^2+q & 3q^2+5q+1 & 3q+2 & 1 &  \\
                    \vdots &  &  &  & \ddots \\
                 \end{array}
  \right],$$
  where $E_{n+1,k}=E_{n,k-1}+(q+1)\,E_{n,k}+q\,E_{n,k+1}$ and $E_{n+1,0}=q E_{n,0}+ q E_{n,1}$. The numbers in the $0$th column are
the Narayana polynomials \cite{Aig01}.
\end{ex}

If $f_k\equiv1$ for $k\geq0$, then it can be reduced to the
recursive matrix and $A_{n,0}(q)$ are called the Catalan-like
numbers, see Aigner~\cite{Aig991,Aig01}. Many combinatorial and
algebraic properties of the triangular array $[A_{n,k}(q)]_{n, k\geq
0}$ have been found. For instance, Aigner
\cite{Aig98,Aig991,Aig01,Aig08} researched various combinatorial
properties of recursive matrices and Hankel matrices of the
Catalan-like numbers. Chen, Liang and Wang \cite{CLW15} considered
the total positivity of recursive matrices. Zhu \cite{Zhu13} gave a
criterion for the strong $q$-log-convexity of the first column of
$[A_{n,k}(q)]_{n,k\geq0}$. However, there is no result for the
strong $q$-log-concavity of the sequence in each row of
$[A_{n,k}(q)]_{n,k\geq0}$. This is our another motivation. In \S4,
we will present some criterions for the strong $q$-log-concavity of
the sequence in each row of $[A_{n,k}(q)]_{n, k\geq0}$.

The remainder part of this paper is arranged as follows. In \S2 and
\S3, we give the proofs of Theorem \ref{thm+binom} and
\ref{thm+three}, respectively. In \S4 we present sufficient
conditions for the strong $q$-log-concavity of rows of the certain
triangle. As applications, we show the log-concavity and the strong
$q$-log-concavity of sequences of rows in many well-known triangular
arrays, such as the Catalan triangles of Aigner and Shaprio, the
large Schr\"oder triangle, the Motzkin triangle, the Bell
polynomials triangle, the Eulerian polynomials triangle and the
Narayana polynomials triangle, and so on, in a unified approach.

\section{Proof of Theorem \ref{thm+binom}}
\begin{proof}
Since a nonnegative sequence $(a_n)_{n\geq0}$ with no internal zeros
is log-concave if and only if $a_ia_j\geq a_{i+1}a_{j-1}$ for $i\geq
j$, the sequence sequence $(a_n)_{n\geq0}$ is strongly
$q$-log-concave. Thus, the second part immediately follows from the
first part for $q=0$. So we only need to show the first part, {\it
i.e.} the linear transformation
$$y_n(q)=\sum_{k=0}^n\binom{a+n}{b+k}x_k(q)$$ preserves the strong $q$-log-concavity
property. We first need the next fact.
\begin{f}\label{sum}
If $(x_n(q))_{n\geq0}$ is strongly $q$-log-concave, then so is
$(x_n(q)+x_{n+1}(q))_{n\geq0}$.
\end{f}

\begin{proof}
Since the sequence $(x_n(q))_{n\geq0}$ is strongly $q$-log-concave,
by the definition, we have $x_i(q)x_j(q)\geq_q x_{i-1}(q)x_{j+1}(q)$
for any $j\geq i\geq0$, which implies that
\begin{eqnarray*}
&&\left[x_{j-1}(q)+x_{j-2}(q)\right]\left[x_{i+1}(q)+x_{i}(q)\right]-\left[x_j(q)+x_{j-1}(q)\right]\left[x_i(q)+x_{i-1}(q)\right]\\
&=&\left[x_{j-1}(q)x_{i+1}(q)-x_j(q)x_i(q)\right]+\left[x_{j-2}(q)x_{i}(q)-x_{j-1}(q)x_{i-1}(q)\right]\\
&&+\left[x_{j-2}(q)x_{i+1}(q)-x_{j}(q)x_{i-1}(q)\right]\\
&\leq_q&0
\end{eqnarray*}
for any $i\geq j\geq0,$ as desired. This proves the fact.
\end{proof}
In the following, we will prove the desired result by induction on
$n$.

If $0\leq n \leq 3$, then we have
\begin{eqnarray}
y_0 (q)&=&x_0(q)\binom{a}{b},\label{0}\\
y_1(q)&=& x_0(q)\binom{a+1}{b} + x_1(q)\binom{a+1}{b+1},\label{1}\\
y_2(q)& =& x_0(q) \binom{a+2}{b}+ x_1(q)\binom{a+2}{b+1}+
x_2(q)\binom{a+2}{b+2},\label{2}\\
y_3(q)& =& x_0(q) \binom{a+3}{b}+ x_1(q)\binom{a+3}{b+1}+
x_2(q)\binom{a+3}{b+2}+x_3(q)\binom{a+3}{b+3},\label{3}
\end{eqnarray}
since $y_n(q)=\sum_{k=0}^n\binom{n+a}{k+b}x_k(q)$. Hence, by
(\ref{0})-(\ref{2}) we obtain that
\begin{eqnarray}\label{eq+}
y_1^2(q)-y_2(q)y_0(q)
&=&x_0^2(q)\left[\binom{a+1}{b}^2-\binom{a}{b}\binom{a+2}{b}\right]\nonumber\\
&&+\left[x_1^2(q)\binom{a+1}{b+1}^2-x_0(q)x_2(q)\binom{a}{b}\binom{a+2}{b+2}\right]\nonumber\\
&&+x_0(q)x_1(q)\left[2\binom{a+1}{b}\binom{a+1}{b+1}-\binom{a}{b}\binom{a+2}{b+1}\right].
\end{eqnarray}
Thus we deduce for $a<b$ that $$y_1^2(q)-y_2(q)y_0(q)\geq_q 0.$$ On
the other hand, by (\ref{eq+}) we also have
$$y_1^2(q)-y_2(q)y_0(q)\geq_q0$$ for $a\geq b$ since
\begin{eqnarray*}
&&\binom{a+1}{b}^2-\binom{a}{b}\binom{a+2}{b}=\binom{a+1}{b-1}\binom{a}{b},\\
&&x_1^2(q)\binom{a+1}{b+1}^2-x_0(q)x_2(q)\binom{a}{b}\binom{a+2}{b+2}\\
&&\geq_q
x_1^2(q)\left[\binom{a+1}{b+1}^2-\binom{a}{b}\binom{a+2}{b+2}\right]=\frac{x_1^2(q)}{b+1}\binom{a+1}{b+2}\binom{a}{b},\\
&&2\binom{a+1}{b}\binom{a+1}{b+1}-\binom{a}{b}\binom{a+2}{b+1}=\frac{a}{a-b+1}\binom{a+1}{b+1}\binom{a}{b}.
\end{eqnarray*}
In the following, we will prove $y_1(q)y_2(q)-y_3(q)y_0(q)\geq_q 0$.
Note for $a\leq b$ that $y_0(q)=0$. So it suffices to consider the
remaining case $a\geq b$. It follows from (\ref{0})-(\ref{3}) we
have
\begin{eqnarray*}
&&y_1(q)y_2(q)-y_3(q)y_0(q)\nonumber\\
&=&\left[x_0(q)\binom{a+1}{b} +
x_1(q)\binom{a+1}{b+1}\right]\left[x_0(q) \binom{a+2}{b}+
x_1(q)\binom{a+2}{b+1}+
x_2(q)\binom{a+2}{b+2}\right]\nonumber\\
&&-x_0(q)\binom{a}{b}\left[ x_0(q) \binom{a+3}{b}+
x_1(q)\binom{a+3}{b+1}+
x_2(q)\binom{a+3}{b+2}+x_3(q)\binom{a+3}{b+3}\right]\nonumber\\
&=&x^2_0(q)\left[\binom{a+1}{b}\binom{a+2}{b}-\binom{a}{b}\binom{a+3}{b}\right]+x_1^2(q)\binom{a+1}{b+1}\binom{a+2}{b+1}\nonumber\\
&&+x_0(q)x_1(q)\left[\binom{a+1}{b}\binom{a+2}{b+1}+\binom{a+1}{b+1}\binom{a+2}{b}-\binom{a}{b}\binom{a+3}{b+1}\right]\nonumber\\
&&+x_0(q)x_2(q)\left[\binom{a+1}{b}\binom{a+2}{b+2}-\binom{a}{b}\binom{a+3}{b+2}\right]\nonumber\\
&&+x_1(q)x_2(q)\binom{a+1}{b+1}\binom{a+2}{b+2}-x_0(q)x_3(q)\binom{a}{b}\binom{a+3}{b+3}\nonumber\\
&\geq_q&x^2_0(q)\left[\binom{a+1}{b}\binom{a+2}{b}-\binom{a}{b}\binom{a+3}{b}\right]\nonumber\\
&&+x_0(q)x_1(q)\left[\binom{a+1}{b}\binom{a+2}{b+1}+\binom{a+1}{b+1}\binom{a+2}{b}-\binom{a}{b}\binom{a+3}{b+1}\right]\nonumber\\
&&+x_0(q)x_2(q)\left[\binom{a+1}{b+1}\binom{a+2}{b+1}+\binom{a+1}{b}\binom{a+2}{b+2}-\binom{a}{b}\binom{a+3}{b+2}\right]\nonumber\\
&&+x_1(q)x_2(q)\left[\binom{a+1}{b+1}\binom{a+2}{b+2}-\binom{a}{b}\binom{a+3}{b+3}\right]\\
&\geq_q&0
\end{eqnarray*}
because $(x_n(q))_{n\geq0}$ is strongly $q$-log-concave and the
following equalities
\begin{eqnarray*}
&&\binom{a+1}{b}\binom{a+2}{b}-\binom{a}{b}\binom{a+3}{b}=\frac{2b}{(a+1-b)(a+3-b)}\binom{a}{b}\binom{a+2}{b},\\
&&\binom{a+1}{b}\binom{a+2}{b+1}+\binom{a+1}{b+1}\binom{a+2}{b}-\binom{a}{b}\binom{a+3}{b+1}=\frac{a(a-b)+a+b}{(a+1-b)(a+2-b)}\binom{a}{b}\binom{a+2}{b+1},\\
&&\binom{a+1}{b+1}\binom{a+2}{b+1}+\binom{a+1}{b}\binom{a+2}{b+2}-\binom{a}{b}\binom{a+3}{b+2}=\frac{ab+2a-b}{(1+b)(2+b)}\binom{a}{b}\binom{a+2}{b+1},\\
&&\binom{a+1}{b+1}\binom{a+2}{b+2}-\binom{a}{b}\binom{a+3}{b+3}=\frac{2(a-b)}{(b+1)(b+3)}\binom{a}{b}\binom{a+2}{b+2}.
\end{eqnarray*}
 Thus, we
obtain that $y_0(q),y_1(q),y_2(q),y_3(q)$ is strongly
$q$-log-concave. So we proceed to the inductive step ($n\geq 4$).

Note that
\begin{eqnarray*}
y_n(q)&=&\sum_{k=0}^n\binom{a+n}{b+k}x_k(q)\\
&=&\sum_{k=0}^{n-1}\binom{a+n-1}{b+k}[x_k(q)+x_{k+1}(q)].
\end{eqnarray*}
Thus by the induction hypothesis and the strong $q$-log-concavity of
$(x_k(q)+x_{k+1}(q))_{k\ge 0}$ by Fact \ref{sum}, we have
$y_0(q),y_1(q),y_2(q),\ldots, y_n(q)$ is strongly $q$-log-concave.
This completes the proof.
\end{proof}
\begin{rem}
If we only consider the log-concavity, we don't need to prove
$y_1y_2-y_0y_3\geq0$ in induction base. Thus, it is obvious that our
proof is much simpler.
\end{rem}

\section{Proof of Theorem \ref{thm+three}}
\begin{proof}
The proof for the second part is similar to that of the first.
Therefore, for brevity, we only show that the linear transformation
$$y_n=\sum_{i=0}^{kn}\binom{n,k}{j}x_i$$ preserves the log-concavity
property.
We first prove the next fact.
\begin{f}\label{q+sum}
If $(x_n)_{n\geq0}$ is log-concave, then so is
$(\sum_{i=0}^kb_ix_{n+i})_{n\geq0}$.
\end{f}

\begin{proof}
Assume that $z_n= \sum_{i=0}^kb_ix_{n+i}$. Let
$$B=[b_{j-i}]_{i,j\geq0}=\left[\begin{array}{ccccccc}
b_0&b_{1}&b_{2}&\ldots&b_{k}&0&\ldots\\
0&b_0&b_{1}&\ldots&b_{k-1}&b_{k}&\ldots\\
0&0&b_0&\ldots&b_{k-2}&b_{k-1}&\ldots\\
\vdots&\vdots&\vdots&\vdots&\vdots&\vdots&\ldots
\end{array}\right]$$
and $$X=[x_{j+i}]_{i,j\geq0}=\left[\begin{array}{ccccccc}
x_0&x_{1}&x_{2}&x_{3}&x_{4}&x_5&\ldots\\
x_{1}&x_{2}&x_{3}&x_{4}&x_5&\ldots&\ldots\\
x_{2}&x_{3}&x_{4}&x_5&\ldots&\ldots&\ldots\\
x_{3}&x_{4}&x_5&\ldots&\ldots&\ldots&\ldots\\
x_{4}&x_5&\ldots&\ldots&\ldots&\ldots&\ldots\\
\vdots&\vdots&\vdots&\vdots&\vdots&\vdots&\ldots
\end{array}\right].$$
Let $Z=[z_{j+i}]_{i,j\geq0}$. It is clear that $Z=BX$. Note that $B$
is TP$_2$ since $(b_i)_{0\leq i\leq k}$ is log-concave and every
minor of order $2$ of $X$ is nonpositive because $(x_n)_{n\geq0}$ is
log-concave. So the classical Cauchy-Binet Theorem shows that every
minor of order $2$ of the product $Z$ of $B$ and $X$ is nonpositive.
Hence $(z_n)_{n\geq0}$ is log-concave. This completes the proof.
\end{proof}

In the following, we will prove that the linear transformation
$$y_n=\sum_{i=0}^{kn}\binom{n,k}{i}x_i$$ preserves the log-concavity
property by induction on $n$.

If $0\leq n \leq 2$, then by $y_n=\sum_{i=0}^{kn}\binom{n,k}{i}x_i$,
we obtain that
\begin{eqnarray}
y_1^2-y_2y_0
&=&\left[\sum_{i=0}^k\binom{1,k}{i}x_i\right]^2-x_0\sum_{i=0}^{2k}\binom{2,k}{i}x_i\nonumber\\
&=&\sum_{i=0}^{2k}\left[\sum_{j=0}^i\binom{1,k}{i-j}\binom{1,k}{j}x_jx_{i-j}\right]-x_0\sum_{i=0}^{2k}\sum_{j=0}^i\binom{1,k}{i-j}\binom{1,k}{j}x_i\nonumber\\
&\geq&\sum_{i=0}^{2k}\left[\sum_{j=0}^i\binom{1,k}{i-j}\binom{1,k}{j}x_0x_{i}\right]-\sum_{i=0}^{2k}\sum_{j=0}^i\binom{1,k}{i-j}\binom{1,k}{j}x_ix_0\nonumber\\
&\geq&0
\end{eqnarray}
because $x_jx_{i-j}\geq x_0x_i$ from the log-concavity of
$(x_n)_{n\geq0}$.

 Thus, we get that $y_0,y_1,y_2$ is log-concave. So we proceed to the
inductive step ($n\geq 3$).

Note that
\begin{eqnarray*}
y_n&=&\sum_{i=0}^{kn}\binom{n,k}{i}x_i\\
&=&\sum_{i=0}^{kn}\binom{n-1,k}{i}\sum_{j=0}^kb_jx_{i+j}.
\end{eqnarray*}
Thus by the induction hypothesis and the log-concavity of
$(\sum_{j=0}^kb_jx_{i+j})_{i\ge 0}$ by means of Fact \ref{q+sum}, we
have the sequence $(y_n)_{n\geq0}$ is log-concave. This completes
the proof.
\end{proof}
\section{Strong $q$-log-concavity of rows in combinatorial triangles}
In this section, we will give some sufficient conditions for the
strong $q$-log-concavity of rows in certain triangular arrays.
\begin{thm}\label{thm+q+log-concave}
 Assume that the infinite lower triangular array
$[A_{n,k}(q)]_{n,k\geq 0}$ satisfies the recurrence
\begin{eqnarray}\label{re}
A_{n,k}(q)=f_{n,k}(q)\,A_{n-1,k-1}(q)+g_{n,k}(q)\,A_{n-1,k}(q)+h_{n,k}(q)\,A_{n-1,k+1}(q)
\end{eqnarray}
for $n\geq 1$ and $k\geq 0$, where $A_{0,0}(q)=1$,
$A_{0,k}(q)=A_{0,-1}(q)=0$ for $k>0$. Assume that three sequences of
polynomials with nonnegative coefficients $(f_{n,k}(q))_{k\geq0}$,
$(g_{n,k}(q))_{k\geq0}$ and $(h_{n,k}(q))_{k\geq0}$ are strongly
$q$-log-concave in $k$, respectively. If
\begin{eqnarray*}
&&f_{n,l}(q)g_{n,k}(q)+f_{n,k}(q)g_{n,l}(q)\geq_q f_{n,l+1}(q)g_{n,k-1}(q)+f_{n,k-1}(q)g_{n,l+1}(q),\\
&&g_{n,l}(q)h_{n,k}(q)+g_{n,k}(q)h_{n,l}(q)\geq_q g_{n,l+1}(q)h_{n,k-1}(q)+g_{n,k-1}(q)h_{n,l+1}(q),\\
&&f_{n,l}(q)h_{n,k}(q)+f_{n,k}(q)h_{n,l}(q)\geq_qf_{n,k-1}(q)h_{n,l+1}(q)+f_{n,l+1}(q)h_{n,k-1}(q)\\
\text{and}&&g_{n,k}(q)g_{n,l}(q)\geq_q f_{n,l+1}(q)h_{n,k-1}(q)
\end{eqnarray*}
for all $1\leq k\leq l\leq n$, then, for any fixed $n$, the row
sequence $(A_{n,k}(q))_{0\leq k\leq n}$ is strongly $q$-log-concave
in $k$. In particular, each row sequence $(A_{n,k}(q))_{0\leq k\leq
n}$ is log-concave in $k$ for any fixed $q\geq0$.
\end{thm}

\begin{proof}
In order to prove that $(A_{n,k}(q))_{0\leq k\leq n}$ is strongly
$q$-log-concave in $k$, we only need to prove
$$A_{n,l}(q)A_{n,k}(q)- A_{n,l+1}(q)A_{n,k-1}(q)\geq_q0$$ for any $l\geq k$, which will be done by induction
on $n$. It is obvious for $n=0$. So, we assume that it follows for
$n\leq m-1$.  For brevity, we write $f_k$ (resp. $g_k$, $h_k$) for
$f_{n,k}(q)$ (resp. $g_{n,k}(q)$, $h_{n,k}(q)$). Then for $n=m$ and
$1\leq k\leq l\leq m$, by (\ref{re}), we have
\begin{eqnarray}\label{eq+main}
&&A_{m,l}(q)A_{m,k}(q)- A_{m,l+1}(q)A_{m,k-1}(q)\nonumber\\
&=&\left[f_{l}A_{m-1,l-1}(q)+g_{l}A_{m-1,l}(q)+h_{l}A_{m-1,l+1}(q)\right]\times\nonumber\\
&&\left[f_{k}A_{m-1,k-1}(q)+g_{k}A_{m-1,k}(q)+h_{k}A_{m-1,k+1}(q)\right]-\nonumber\\
&&\left[f_{l+1}A_{m-1,l}(q)+g_{l+1}A_{m-1,l+1}(q)+h_{l+1}A_{m-1,l+2}(q)\right]\times\nonumber\\
&&\left[f_{k-1}A_{m-1,k-2}(q)+g_{k-1}A_{m-1,k-1}(q)+h_{k-1}A_{m-1,k}(q)\right]\nonumber\\
&=&\underbrace{f_{l}f_{k}A_{m-1,l-1}(q)A_{m-1,k-1}(q)-f_{l+1}f_{k-1}A_{m-1,l}(q)A_{m-1,k-2}(q)}+\nonumber\\
&&f_{l}g_kA_{m-1,l-1}(q)A_{m-1,k}(q)-f_{k-1}g_{l+1}A_{m-1,l+1}(q)A_{m-1,k-2}(q)+\nonumber\\
&&[f_{k}g_{l}-f_{l+1}g_{k-1}]A_{m-1,l}(q)A_{m-1,k-1}(q)+\nonumber\\
&&\underbrace{h_{l}h_{k}A_{m-1,l+1}(q)A_{m-1,k+1}(q)-h_{l+1}h_{k-1}A_{m-1,l+2}(q)A_{m-1,k}(q)}+\nonumber\\
&&h_{k}g_{l}A_{m-1,l}(q)A_{m-1,k+1}(q)-h_{l+1}g_{k-1}A_{m-1,l+2}(q)A_{m-1,k-1}(q)+\nonumber\\
&&[h_{l}g_{k}-h_{k-1}g_{l+1}]A_{m-1,l+1}(q)A_{m-1,k}(q)+\nonumber\\
&&\underbrace{g_{l}g_{k}A_{m,l}(q)A_{m,k}(q)-g_{l+1}g_{k-1}A_{m-1,l+1}(q)A_{m-1,k-1}(q)}+\nonumber\\
&&f_{l}h_{k}A_{m-1,l-1}(q)A_{m-1,k+1}(q)-h_{l+1}f_{k-1}A_{m-1,l+2}(q)A_{m-1,k-2}(q)+\nonumber\\
&&[h_{l}f_{k}A_{m-1,l+1}(q)A_{m-1,k-1}(q)-f_{l+1}h_{k-1}A_{m-1,l}(q)A_{m-1,k}(q)].
\end{eqnarray}
In what follows we will prove the nonnegativity of (\ref{eq+main})
in $q$.

Firstly, it follows from the strong $q$-log-concavities of
$(A_{m-1,k}(q))_{0\leq k\leq m-1}$ and $(f_k(q))_{k\geq0}$ that
 \begin{eqnarray}\label{ieq1}A_{m-1,l-1}(q)A_{m-1,k-1}(q)-A_{m-1,l}(q)A_{m-1,k-2}(q)\geq_q0\nonumber\end{eqnarray}
 and \begin{eqnarray}\label{ieq1}f_{l}f_{k}-f_{l+1}f_{k-1}\geq_q0,\nonumber\end{eqnarray}
which implies
\begin{eqnarray}\label{ieq1}f_{l}f_{k}A_{m-1,l-1}(q)A_{m-1,k-1}(q)-f_{l+1}f_{k-1}A_{m-1,l}(q)A_{m-1,k-2}(q)\geq_q0.\end{eqnarray}
Similarly, we also have
\begin{eqnarray}\label{ieq2}h_{l}h_{k}A_{m-1,l+1}(q)A_{m-1,k+1}(q)-h_{l+1}h_{k-1}A_{m-1,l+2}(q)A_{m-1,k}(q)\geq_q0.\end{eqnarray}

Secondly, by the strong $q$-log-concavity of $(A_{m-1,k}(q))_{0\leq
k\leq m-1}$ and $$f_{l}g_{k}+f_{k}g_{l}\geq_q
f_{l+1}g_{k-1}+f_{k-1}g_{l+1},$$ we get
\begin{eqnarray}\label{ieq4}
&&f_{l}g_kA_{m-1,l-1}(q)A_{m-1,k}(q)-f_{k-1}g_{l+1}A_{m-1,l+1}(q)A_{m-1,k-2}(q)+\nonumber\\
&&[f_{k}g_{l}-f_{l+1}g_{k-1}]A_{m-1,l}(q)A_{m-1,k-1}(q)\nonumber\\
&\geq_q&[f_{l}g_k+f_{k}g_{l}-f_{l+1}g_{k-1}]A_{m-1,l}(q)A_{m-1,k-1}(q)-f_{k-1}g_{l+1}A_{m-1,l+1}(q)A_{m-1,k-2}(q)\nonumber\\
&\geq_q&[f_{l}g_k+f_{k}g_{l}-f_{l+1}g_{k-1}-f_{k-1}g_{l+1}]A_{m-1,l}(q)A_{m-1,k-1}(q)\nonumber\\
&\geq_q&0.
\end{eqnarray}
In a similar way, we also have
\begin{eqnarray}\label{ieq5}
&&h_{k}g_{l}A_{m-1,l}(q)A_{m-1,k+1}(q)-h_{l+1}g_{k-1}A_{m-1,l+2}(q)A_{m-1,k-1}(q)+\nonumber\\
&&[h_{l}g_{k}-h_{k-1}g_{l+1}]A_{m-1,l+1}(q)A_{m-1,k}(q)\nonumber\\
&\geq_q&[h_{k}g_{l}+h_{l}g_{k}-h_{k-1}g_{l+1}-h_{l+1}g_{k-1}]A_{m-1,l+1}(q)A_{m-1,k}(q)\nonumber\\
&\geq_q&0.
\end{eqnarray}
Finally,  it follows from the strong $q$-log-concavity of
$(g_k(q))_{k\geq0}$ that
\begin{eqnarray}\label{ieq1}g_{l}g_{k}-g_{l+1}g_{k-1}\geq_q0.\nonumber\end{eqnarray}
Thus we have
\begin{eqnarray}\label{ieq6}
&&\underbrace{g_{l}g_{k}A_{m-1,l}(q)A_{m-1,k}(q)-g_{l+1}g_{k-1}A_{m-1,l+1}(q)A_{m-1,k-1}(q)}+\nonumber\\
&&f_{l}h_{k}A_{m-1,l-1}(q)A_{m-1,k+1}(q)-h_{l+1}f_{k-1}A_{m-1,l+2}(q)A_{m-1,k-2}(q)+\nonumber\\
&&[h_{l}f_{k}A_{m-1,l+1}(q)A_{m-1,k-1}(q)-f_{l+1}h_{k-1}A_{m-1,l}(q)A_{m-1,k}(q)]\nonumber\\
&\geq_q&\underbrace{g_{l}g_{k}[A_{m-1,l}(q)A_{m-1,k}(q)-A_{m-1,l+1}(q)A_{m-1,k-1}(q)]}+\nonumber\\
&&f_{l}h_{k}A_{m-1,l-1}(q)A_{m-1,k+1}(q)-h_{l+1}f_{k-1}A_{m-1,l+2}(q)A_{m-1,k-2}(q)+\nonumber\\
&&\underbrace{[f_{l+1}h_{k-1}+h_{l+1}f_{k-1}-f_{l}h_{k}]}A_{m-1,l+1}(q)A_{m-1,k-1}(q)-\nonumber\\
&&f_{l+1}h_{k-1}A_{m-1,l}(q)A_{m-1,k}(q)]\nonumber\\
&=&\underbrace{g_{l}g_{k}[A_{m-1,l}(q)A_{m-1,k}(q)-A_{m-1,l+1}(q)A_{m-1,k-1}(q)]}+\nonumber\\
&&\underbrace{f_{l+1}h_{k-1}[A_{m-1,l+1}(q)A_{m-1,k-1}(q)-A_{m-1,l}(q)A_{m-1,k}(q)]}+\nonumber\\
&&f_{l}h_{k}[A_{m-1,l-1}(q)A_{m-1,k+1}(q)-A_{m-1,l+1}(q)A_{m-1,k-1}(q)]+\nonumber\\
&&\underbrace{h_{l+1}f_{k-1}[A_{m-1,l+1}(q)A_{m-1,k-1}(q)-A_{m-1,l+2}(q)A_{m-1,k-2}(q)]}\nonumber\\
&\geq_q&\underbrace{[g_{l}g_{k}-f_{l+1}h_{k-1}][A_{m-1,l}(q)A_{m-1,k}(q)-A_{m-1,l+1}(q)A_{m-1,k-1}(q)]}\nonumber\\
&\geq_q&0
\end{eqnarray}
since
$$f_{l}h_{k}+f_{k}h_{l}\geq_qf_{k-1}h_{l+1}+f_{l+1}h_{k-1}$$
and $$g_kg_l\geq_q f_{l+1}h_{k-1}.$$

Thus, by (\ref{eq+main})- (\ref{ieq6}), we get
$$A_{m,k}(q)A_{m,l}(q)- A_{m,l+1}(q)A_{m,k-1}(q)\geq_q 0$$
for $1\leq k\leq l\leq m$. The proof is complete.
\end{proof}

The following special case related to the Riordan array
\cite{MRSV97} may be more interesting.
\begin{prop}
Define the matrix $[A_{n,k}(q)]_{n,k\geq0}$ recursively:
\begin{eqnarray*}
A_{0,0}(q)&=&1, \quad A_{0,k}(q)=0\quad(k>0),\\
A_{n,0}(q)&=&e(q)\,A_{n-1,0}(q)+h(q)\,A_{n-1,1}(q),\\
A_{n,k}(q)&=&A_{n-1,k-1}(q)+g(q)\,A_{n-1,k}(q)+h(q)\,A_{n-1,k+1}(q)\quad
(n,k\geq 1).
\end{eqnarray*}
If $e(q)g(q)\geq_q h(q) \geq_q 0$ and $g(q)\geq_q e(q)\geq_q 0$,
then each row $(A_{n,k}(q))_{0\leq k\leq n}$ is strongly
$q$-log-concave.
\end{prop}

Applying Theorem \ref{thm+q+log-concave} to combinatorial arrays in
Example \ref{exm}, we have the following results in a unified
manner.
\begin{cor}
Each row sequence of the Bell polynomials triangle, the Eulerian
polynomials triangle and the Narayana polynomials triangle is
strongly $q$-log-concave, respectively.
\end{cor}
\begin{cor}
Each row sequence in the Catalan triangles of Aigner and Shaprio,
the Motzkin triangle, the large Schr\"oder triangle, and the Bell
triangle is log-concave, respectively.
\end{cor}


\end{document}